\documentclass{aims}
\usepackage{amsmath}
  \usepackage{paralist}
  \usepackage{graphics} 
  \usepackage{epsfig} 
\usepackage{graphicx}  \usepackage{epstopdf}
 \usepackage[colorlinks=true]{hyperref}
\hypersetup{urlcolor=blue, citecolor=red}

\def\currentvolume{X}
 \def\currentissue{X}
  \def\currentyear{200X}
   \def\currentmonth{XX}
    \def\ppages{X--XX}
     \def\DOI{10.3934/xx.xx.xx.xx}

\newtheorem{theorem}{Theorem}[section]

\newtheorem{lemma}[theorem]{Lemma}
\newtheorem{proposition}{Proposition}

\theoremstyle{definition}

\newtheorem{remark}{Remark}

\title[Normal Boundary Control of Stokes Equations] 
      {Null Controllability of the Incompressible Stokes Equations in a 2-D Channel Using Normal Boundary Control}

\author[Shirshendu Chowdhury and Debanjana Mitra and Michael Renardy]{}
\email{shirshendu@iiserkol.ac.in}
 \email{deban@math.iitb.ac.in}
 \email{mrenardy@math.vt.edu}

\thanks{Shirshendu Chowdhury acknowledges financial support from an INSPIRE Fellowship. 
Michael Renardy and Debanjana Mitra acknowledge support from 
the National Science Foundation under
Grant DMS-1514576.}

\thanks{$^*$ Corresponding author: Michael Renardy}

\begin{document}
\maketitle

\centerline{\scshape Shirshendu Chowdhury}
\medskip
{\footnotesize
 \centerline{Indian Institute of Science Education and Research}
   \centerline{Kolkata, West Bengal, India}
} 

\medskip

\centerline{\scshape Debanjana Mitra}
\medskip
{\footnotesize
 \centerline{Department of Mathematics}
 \centerline{Indian Institute of Technology, Bombay}
 \centerline{Powai, Mumbai, Maharashtra 400076, India}
} 

\medskip
\centerline{\scshape Michael Renardy$^*$}
\medskip
{\footnotesize
 \centerline{Department of Mathematics}
   \centerline{Virginia Tech}
   \centerline{Blacksburg, VA 24061-0123, USA}
}

\bigskip

 \centerline{(Communicated by the associate editor name)}

\newcommand{\R}        {\mathbb {R}}
\newcommand{\C}        {\mathbb {C}}
\newcommand{\h}        {\mathcal {H}}
\newcommand{\Z}        {\mathbb {Z}}
\newcommand{\I}        {\infty}
\newcommand{\<}        {\leq}
\newcommand{\lap}        {\triangle}
\newcommand{\grad}        {\nabla}
\newtheorem{defin}[theorem]{Definition}
\def\N{{\mathbb N}}
\newcommand{\notes}[1]{{\color{red} {\tt \footnotesize #1}}}
\newcommand{\qtq}[1]{\qquad \textrm{#1} \qquad}
\newcommand{\be} {\begin{equation}}
\newcommand{\ee} {\end{equation}}
\newcommand{\bea} {\begin{eqnarray}}
\newcommand{\eea} {\end{eqnarray}}
\newcommand{\Bea} {\begin{eqnarray*}}
\newcommand{\Eea} {\end{eqnarray*}}
\newcommand{\pa} {\partial}
\newcommand{\ov} {\over}
\newcommand{\al} {\alpha}
\newcommand{\ba} {\beta}
\newcommand{\de} {\delta}
\newcommand{\ga} {\gamma}
\newcommand{\Ga} {\Gamma}
\newcommand{\Om} {\Omega}
\newcommand{\om} {\omega}
\newcommand{\De} {\Delta}
\newcommand{\la} {\lambda}
\newcommand{\si} {\sigma}
\newcommand{\Si} {\Sigma}
\newcommand{\La} {\Lambda}
\newcommand{\Th} {\Theta}
\newcommand{\noi} {\noindent}
\newcommand{\lab} {\label}
\newcommand{\na} {\nabla}
\newcommand{\va} {\varphi}
\newcommand{\var} {\varepsilon}
\numberwithin{equation}{section}
\newtheorem{claim}[theorem]{Claim}

\begin{abstract}

In this paper, we consider the Stokes equations in a two-dimen\-sional channel with periodic
conditions in the direction of the channel. We establish null controllability of this system 
 using a boundary control which 
 acts on the normal component of the velocity only.
We show null 
controllability of the system, subject to a constraint of zero average,
by proving an observability inequality with the help of 
a M\"{u}ntz-Sz\'{a}sz Theorem.

\end{abstract}

\section{Introduction}\label{sec1}

We consider a viscous incompressible fluid flow in a two dimensional periodic channel,
defined by $(x,y) \in (-\infty,\infty)\times[0, 1]$, with the walls located at $y = 0$ and $y = 1$. So the boundary
of the channel is split into two parts, namely the upper and the lower part.
The Navier-Stokes equations for a viscous incompressible fluid for  $(x,y)\in (-\infty,\infty)\times (0,1) \subset \R^2 $  and $t\in (0,T)$
are
\begin{eqnarray}\label{eq:2D}
\frac{\partial U}{\partial t}(x,y,t)+(U(x,y,t).\grad)U(t,x,y)-\nu\lap U(x,y,t)\, = \, \grad p(x,y,t),\nonumber\\
 \quad\quad\quad \forall\, (x,y,t)\in (-\infty,\infty)\times (0,1)\times (0,T),\nonumber\\
\mbox{div}\ U(x,y,t)\, = \, 0, \quad \forall\, (x,y,t)\in (-\infty,\infty)\times (0,1)\times (0,T), \nonumber\\
U(x,0,t)=(0,0)\ ,\ U(x,1,t)=(0,\psi(x,t)), \quad  \forall \, (x,t)\in (-\infty,\infty)\times (0,T), \\
U(x+L,y,t) \, = \, U(x,y,t), \quad p(x+L,y,t) \, = \, p(x,y,t),\nonumber\\ \quad \forall\, (x,y,t)\in (-\infty,\infty)\times (0,1)\times (0,T),\nonumber\\
U(x,y,0)=(u_0(x,y),\ v_0(x,y)), \ \ \forall\ (x,y)\in (-\infty,\infty)\times(0,1),\nonumber
\end{eqnarray}
where $U(x,y,t)=(u(x,y,t),v(x,y,t))$ denotes the velocity of
the fluid in $\R^2$, $p(x,y,t)$ denotes the pressure and $\psi$ is the boundary control. The viscosity coefficient $\nu$
is assumed to be a positive constant and $L$ is any positive number.
A simple stationary 
solution (with $\psi=0$) of system \eqref{eq:2D} is given by $(U,p) = (0,0,0)$, corresponding to a fluid at rest.

We assume that both the velocity field $U= (u,v)$ and the pressure $p$ are $L$ periodic in the
first spatial coordinate $x$ .

Here we consider the following linearized system of \eqref{eq:2D}.
\begin{eqnarray}\label{eq:1.2} 
\frac{\partial U}{\partial
t}-\nu\lap U &=&\grad p\quad \mathrm{in}\quad (-\infty,\infty)\times (0,1)\times (0,T),\nonumber\\
\mbox{div}\ U&=&0 \quad \mathrm{in}\quad (-\infty,\infty)\times (0,1)\times (0,T),\nonumber\\
U(x,0,t)&=&(0,0) \quad \forall \, (x,t)\in (-\infty,\infty)\times (0,T),\nonumber\\
U(x,1,t)&=&(0,\psi(x,t)), \quad \forall \, (x,t)\in  (-\infty,\infty)\times (0,T),\\
U(x+L,y,t)&=&U(x,y,t),\quad \forall\, (x,y,t)\in (-\infty,\infty)\times (0,1)\times (0,T),\nonumber\\
p(x+L,y,t)&=&p(x,y,t),\quad \forall\, (x,y,t)\in (-\infty,\infty)\times (0,1)\times (0,T),\nonumber\\
U(x,y,0)&=&U_0(x,y),\quad \forall \, (x,y)\in (-\infty,\infty)\times (0,1)\times (0,T).\nonumber
\end{eqnarray}
\begin{defin}
The system \eqref{eq:1.2}  is null controllable in a Hilbert space $Z$ at time $T > 0$, if for
any initial condition $U_0 \in Z$, there exists a control $\psi$ such that the
solution $U$ of  \eqref{eq:1.2} with control $\psi$ hits 0 at time T , i.e.
$U(T) = 0.$
\end{defin} 
Our goal in this paper is to study the null controllability of the  
linearized system \eqref{eq:1.2} by using a control $\psi$ acting only in the 
normal direction  on the upper part of boundary. 
Due to the incompressibility condition  div $U = 0$, 
we necessarily have
\begin{equation}\label{eqcontrl+avg}
	\int_{0}^{L} \psi(x,t) \, dx = 0,\ \forall t\in (0,T).
\end{equation}

The main obstacle to null controllability using only one control acting in the normal direction of the boundary is as follows.  
In \eqref{eq:1.2}, we denote $U(x,y,0)=U_0(x,y)=(u_0,v_0)$ and $U=(u,v)$.
Taking a dot product between $\eqref{eq:1.2}_1$ and $(\sin(n\pi y),0)$ and then using an integration by parts on $(0,L)\times (0,1)$
and the condition (\ref{eqcontrl+avg}),
we get 
$${d\over dt}\int_0^L\int_0^1 u(x,y,t)\sin(n\pi y)\,dy\,dx=-\nu n^2\pi^2\int_0^L\int_0^1 u(x,y,t)\sin(n\pi y)\,dy\,dx.$$
Consequently,
$$\int_{0}^{L}\int_{0}^{1}u(x,y,T)\sin (n\pi y)\,dy\,dx=e^{-\nu(n\pi)^2 T}\int_{0}^{L}\int_{0}^{1}u_0(x,y)\sin (n\pi y)\,dy\,dx.$$ 
Thus $U(x,y,T)=(0,0)$ implies that the initial condition $U_0(x,y)$ has to satisfy 
\begin{equation}
\int_{0}^{L}\int_{0}^{1} u_{0}(x,y)\sin (n\pi y)\,dy\,dx=0, \quad \forall \, n\in \N.
\label{remintromiss}
\end{equation}
So there are infinitely many directions, namely $(\sin (l\pi y),0),l\in \N$,
which are not null controllable by the control acting in the normal direction of the boundary. 
Note that the subspace spanned by these directions is 
infinite-dimensional. The question then arises whether the system is null controllable in the orthogonal complement
of this subspace, and this paper will answer this question affirmatively.

In view of this discussion and \eqref{eqcontrl+avg}, we study the null controllability of the 
linearized system by using the control acting in the normal direction of the boundary
with appropriate constraints on the control and initial condition,
i.e. 
\begin{eqnarray}\label{eq:LS}
\frac{\partial U}{\partial
t}-\nu\lap U \,=\,\grad p \quad \mathrm{in}\quad (-\infty,\infty)\times (0,1)\times (0,T)\nonumber\\
\mbox{div}\ U \, = \,0 \quad \mathrm{in}\quad (-\infty,\infty)\times (0,1)\times (0,T),\nonumber\\
U(x,0,t)\,=\, (0,0), \quad \forall\, (x,t)\in (-\infty,\infty)\times (0,T),\nonumber\\
U(x,1,t)\,=\,(0,\psi(x,t)), \quad \forall\, (x,t)\in (-\infty,\infty)\times (0,T),\nonumber\\
\quad \int_0^L \psi(x,t)\, dx =0, \quad \forall \, t\in (0,T),\\
U(x+L,y,t)\,= \,U(x,y,t), \quad \forall\, (x,y,t)\in (-\infty,\infty)\times (0,1)\times (0,T), \nonumber\\
p(x+L,y,t)\,=\,p(x,y,t),\quad \forall\, (x,y,t)\in (-\infty,\infty)\times (0,1)\times (0,T),\nonumber\\
U(x,y,0)\,=\,U_0(x,y)=(u_0(x,y),v_0(x,y)), \quad \forall\, (x,y)\in (-\infty,\infty)\times (0,1),\nonumber\\ 
\quad \int_0^{L} u_0(x,y)\, dx = 0, \quad \forall\; y\in (0,1). \nonumber
\end{eqnarray}

Using spectral methods we prove null controllability for
$U_0$ belonging to  $\mathbf{V}_{\sharp,n}^0(\Omega)$ with 
$\displaystyle \int_0^L u_0(x,y)\,dx= 0$, for all $y\in (0,1)$,
where $ \mathbf{V}_{\sharp,n}^0(\Omega)$ denotes
space of 
 $L$-periodic divergence free $\mathbf{L}^2$ vector functions which have normal trace zero.
(For details of the function spaces, see Section 2.1). This is the main result (see Theorem \ref{thnull} in Section \ref{sec3}) of this paper. 

As far as we know there are no prior null controllability results using one normal boundary control for 
the Stokes system \eqref{eq:1.2}.

The proof of the null controllability result relies on an observability inequality (see Section 3) 
for the solutions of the adjoint system and the spectral analysis of the linearized operator.
The spectrum of the Stokes operator lies on the left side of the complex plane. It
consists of a family of real eigenvalues, which diverges to $-\infty$. 
Moreover,
explicit expressions of eigenvalues and
eigenfunctions in terms of a Fourier basis are obtained. This helps to split the  
observability inequality into observability 
inequalities corresponding to each Fourier mode of the adjoint system.
The observability inequality for each mode will be established using a parabolic type of Ingham inequality.  The proof that the
observability inequality (Lemma \ref{lemobsk}) for the $k$th Fourier mode holds  with a positive constant $C_T$, independent of $k$,
is the key result in this work.

The stabilization of the incompressible Navier-Stokes system in a 2-D channel (with periodic
conditions along the $x$ axis) linearized around a steady-state parabolic laminar flow profile $(L(y), 0)$
has been studied in Munteanu \cite{Munteanu}
and Barbu \cite{Barbu1}. In particular,
Munteanu in \cite{Munteanu} proved that the linearized system of (1.1) around $(L(y)=C(y^2-y), 0)$ is exponentially
stabilizable with some decay rate $\omega, 0 < \omega \leq \nu\pi^2$ by a normal boundary, finite-dimensional
feedback controller acting on the upper wall $\Gamma_1 (y = 1)$ only. A similar stability result for this
linearized system when the normal velocity is controlled on the both walls $\Gamma_0 (y = 0)$ and $\Gamma_1$
(y= 1) of the channel is proved by Barbu in \cite{Barbu1}. In \cite{Barbu0}, Barbu established that the exponential
stability of the linearized system around $(L(y), 0)$ can be achieved using a
finite number of Fourier modes and a boundary feedback stochastic controller which acts on
the normal component of velocity only. In the present paper (see Section 2.2), we also notice 
that the incompressible Navier-Stokes system (1.1) linearized around the origin (0, 0) is stabilizable
using only a normal boundary control, with any decay rate $\omega$ such that $0 < \omega \leq \nu\pi^2$. Thus we know that
the exponential decay rate can be at most $\nu \pi^2$ for the linearized system. (See also equation (96)
in Section 9 in \cite{Munteanu} or Section 2 in \cite{chowdhury}). A similar situation occurs also in Triggiani \cite{Tri} for a linearized system where
homogeneous boundary conditions on the tangential component $u$ of the velocity are of Neumann
type.

The boundary stabilization of Navier-Stokes equations, with tangential controllers or normal
controllers was studied in two dimensions, for example by Barbu \cite{Barbu2,Barbu1}, Munteanu \cite{Munteanu2}, Coron \cite{Vaz1},
Krstic \cite{Vaz2},\cite{Krstic1}, \cite{Krstic2}, Raymond \cite{Raymond2}. In most of these papers, either there are
sufficiently many boundary controls so there are no missing directions, or stabilizability is proved but with no specific
decay rate (except in \cite{Raymond2}). In contrast,
in this work we are using only one boundary control (acting on the normal component of velocity) on the
upper part of the boundary and we are looking for null controllability instead of stabilizability. 

 In 
  \cite{chowdhury}, Chowdhury and 
 Ervedoza proved a local stabilization result for the viscous incompressible Navier-Stokes equations \eqref{eq:2D}
  at any exponential decay rate by a normal boundary control acting at the upper boundary.
 The linearized system around zero is exponentially stable with decay rate 
 $\nu\pi^2$ but not stabilizable at a higher decay rate. To overcome this difficulty, 
 the idea is to use the nonlinear term to stabilize the system in the directions which are not stabilizable 
 for the linearized equations. The argument is based on the power series expansion 
 method  introduced by 
 J.-M. Coron and E. Cr\'{e}peau in \cite{coron2} and described in \cite[Chapter 8]{coron1}.

Coron and Guerrero 
in \cite{coron5} consider the two-dimensional Navier-Stokes system in a torus. They establish the local
null controllability with internal controls having one vanishing component. Note that in their case also
the linearized control system around the origin is not null
controllable. In fact for the linearized system infinitely many missing directions are there corresponding to
$$
\lambda_1\sin\frac{2n\pi x}{L}+\lambda_2\cos\frac{2n\pi x}{L}, \quad n\in \Z, \quad \lambda_i\in \R, \quad i=1,2,\quad L>0,$$
like our case here for $sin(n\pi y)$.
But in \cite{coron5} the nonlinear term helps to get this null controllability using the return method. 
Coron and Lissy proved in \cite{coron6} local null controllability for the three-dimensional Navier-Stokes
equations on a smooth bounded domain of $\mathbb{R}^3$ using localized interior control with two vanishing
components. In this case also, the linearized system is not necessarily null controllable even 
if the control is distributed on the entire domain. They show
local null controllability using the return method together with a new algebraic method inspired
by the works of M. Gromov.
For our system also the study of null controllability for the nonlinear system \eqref{eq:2D} is an open question.
Moreover 
null controllability of the linearized and nonlinear system when 
control is localized is an interesting issue. These are interesting challenges for future research.

This paper is organized as follows. In Section \ref{sec2}, we introduce function spaces required for our analysis.
Then we study
the behavior of the spectrum of the linearized operator and well-posedness
of the linearized system \eqref{eq:1.2} and the corresponding adjoint system.
In Section \ref{sec3}, we split the observability inequality into  observability 
inequalities corresponding to each Fourier mode of the adjoint system. 
Thereafter we give the completion of the proof of the observability inequality using  a
M\"{u}ntz-Sz\'{a}sz Theorem for each mode and showing 
uniformity of the constant arising in the observability inequality.
In this fashion, null controllability (Theorem \ref{thnull}) is proved.

\section{Spectral analysis of the Stokes system} \label{sec2}
In this section we study the Stokes
system using its modal description.  In particular, we identify
the modes which are null controllable for the linearized system.

\subsection{Functional framework}
Let
$$\Omega = \left\{(x,y)\in \R^2 : 0 < y < 1\right\}$$
with boundary $\ \Gamma  = \Gamma_0 \cup \Gamma_1$ , where 
$$\Gamma_i=\{(x,y)\in \R^2: y=i\}, i= 0,1$$
and 
$$\ \Omega_L=(0,L)\times(0,1), \quad \mathrm{for}\, L>0.$$

Define $L_{\sharp}^2(\Omega)$
as
$$L_{\sharp}^2(\Omega) = \left\{f\in L_{loc}^2(\Omega) :  f|_{\Omega_{L}}\in L^2(\Omega_{L}), f (x + L, y) = f(x,y)\ \mbox{for a.e.}\ (x, y) \in \Omega\right\}.$$
 
 We also define the space
$$H_{\sharp}^1(\Omega)
 =\{f\in L_{\sharp}^2(\Omega) ;\quad  f|_{\Omega_{L}}\in H^1(\Omega_{L}), f(x,0)=f(x,1)=0,$$
 $$ f (0, y)=f (L, y)\ \mbox{in the trace sense} \}.$$

Let us denote the vector valued functional spaces:
$$ \mathbf{L}^2_\sharp(\Omega)= L^2_\sharp(\Omega)\times L^2_\sharp(\Omega),\quad 
\mathbf{H}^1_\sharp(\Omega)= H^1_\sharp(\Omega)\times H^1_\sharp(\Omega),$$
and $\mathbf{H}^{-1}_\sharp(\Omega)$ is the dual space of $\mathbf{H}^1_\sharp(\Omega)$.

 We now introduce the following spaces of divergence free vector fields:
 $$ \mathbf{V}_{\sharp,n}^0(\Omega)=\left\{U=(u,v) \in \mathbf{L}_{\sharp}^2(\Omega); \quad \mbox{div}\ U=0, \quad U.n=0\ \mbox{on}\ \Gamma\right\}, $$
(here the subscript $n$ indicates the vanishing of the normal component) and 
$$ \mathbf{V}_{\sharp}^1(\Omega)=\left\{U=(u,v) \in \mathbf{H}_{\sharp}^1(\Omega) ;\quad 
\mbox{div}\ U=0 \right\}.$$
We also denote the dual space of $ \mathbf{V}_{\sharp}^1(\Omega)$ by $\mathbf{V}_{\sharp}^{-1}(\Omega)$.

We also define the space of $L^2$ functions in $(0,L)$ having mean-value zero by $\dot{L}^2(0,L)$, i.e.,
$$ \dot{L}^2(0,L)= \{g \in L^2(0,L)\mid \int_0^L g(x)\,dx=0\}.$$

Let us denote by $P$ the Helmholtz orthogonal projection operator from $L_{\sharp}^2(\Omega)\times L_{\sharp}^2(\Omega)$ on to $\mathbf{V}_{\sharp,n}^0(\Omega)$, defined as

$$P(f)=f-\bigtriangledown q,$$
where $q$ is the weak solution of
$$\Delta q=\mbox{div} f, \quad 
{\partial q\over\partial n}=f.n\ \mbox{on}\ \Gamma.$$

Further, taking $\psi=0$ in \eqref{eq:LS}, the linear operator associated to \eqref{eq:LS}  is 
the Stokes operator
$A = \nu P \Delta$ , with domain $\mathcal{D}(A) = H^2(\Omega) \cap \mathbf{V}_{\sharp,0}^1(\Omega)$ in $\mathbf{V}_{\sharp,n}^0(\Omega)$.

The next lemma recalls well known properties of the Stokes operator.
\begin{lemma}\label{analytic}
The operator $(A,\mathcal{D}(A))$ is the infinitesimal generator of a strongly continuous
analytic semigroup $(e^{tA})_{\{t\geq0\}}$ on $\mathbf{V}_{\sharp,n}^0(\Omega)$. 
The operator $(A,\mathcal{D}(A))$ is self-adjoint in $\mathbf{V}_{\sharp,n}^0(\Omega)$, i.e 
$\mathcal{D}(A)=\mathcal{D}(A^{*})$ and $A=A^{*}$. 
\end{lemma}

\subsection{Linearized system and its modal description}
Here we study the linearized system with a normal boundary control $\psi$ and some details of the spectrum of the corresponding linearized operator.

The adjoint problem corresponding to \eqref{eq:LS} is 
\begin{eqnarray}\label{eqadj}
-\frac{\partial \Phi}{\partial
t}\, =\, \nu\lap \Phi + \,\grad q \quad \mathrm{in}\quad (-\infty,\infty)\times (0,1)\times (0,T),\nonumber\\
\mbox{div}\ \Phi \, = \,0 \quad \mathrm{in}\quad (-\infty,\infty)\times (0,1)\times (0,T),\nonumber\\
\Phi(x,0,t)\,=\, (0,0)= \, \Phi(x,1,t), \quad \forall\, (x,t)\in (-\infty,\infty)\times (0,T),\nonumber\\
\Phi(x+L,y,t)\,= \,\Phi(x,y,t), \quad \forall\, (x,y,t)\in (-\infty,\infty)\times (0,1)\times (0,T), \\
q(x+L,y,t)\,=\,q(x,y,t)\quad \forall\, (x,y,t)\in (-\infty,\infty)\times (0,1)\times (0,T),\nonumber\\
\Phi(x,y,T)\,=\,\Phi_T(x,y), \quad \forall\, (x,y)\in (-\infty,\infty)\times (0,1),\nonumber\\ \quad \int_0^{L} \Phi_T(x,y)\, dx = 0, \quad \forall\; y\in (0,1). \nonumber
\end{eqnarray}

Let us consider 
the eigenvalue problem $\lambda\Phi=A\Phi=\nu P\Delta\Phi$ (to limit the number of different symbols, we use the same letters for the eigenfunctions as for the solution of the adjoint), i.e.
\begin{eqnarray}\label{eq:pi}
\lambda \Phi -\nu\lap \Phi &=&\grad q,\quad
\mbox{div}\ \Phi=0,\nonumber\\
\Phi(x,0)&=&(0,0),\quad
\Phi(x,1)=(0,0),\\
\Phi(x+L,y)&=&\Phi(x,y),\quad
q(x+L,y)=q(x,y).\nonumber
\end{eqnarray}

We expand $(\Phi,q)=(\phi,\xi, q)$ into Fourier series:
\begin{eqnarray*}
\phi(x,y)&=&\sum_{k\in \frac{2\pi}{L}\Z}\phi_k(y) e^{ikx},\\
\xi(x,y)&=&\sum_{k\in \frac{2\pi}{L}\Z}\xi_k(y) e^{ikx},\\
q(x,y)&=&\sum_{k\in \frac{2\pi}{L}\Z}q_k(y) e^{ikx}.
\end{eqnarray*}

The eigenvalue problem for $(\phi_k , \xi_k , q_k )$ is
\begin{eqnarray}\label{eqefk}
 (\lambda+\nu k^2)\phi_k(y)-\nu \phi_k^{\prime\prime}(y)&=&ik q_k(y)\nonumber\\
 (\lambda+\nu k^2)\xi_k(y)-\nu \xi_k^{\prime\prime}(y)&=&q_k^{\prime}(y)\\
 ik \phi_k(y)+\xi_k^{\prime}(y)&=&0\nonumber\\
 \phi_k(0)=\phi_k(1)=\xi_k(0)=\xi_k(1)&=&0.\nonumber. 
\end{eqnarray}
The cases $k \neq 0$ and $k = 0$ need to be considered separately.

For $k \neq 0$ we have 
\begin{eqnarray}\label{eqeval1}
&&\nu \xi_k^{iv}(y)-(\lambda+2\nu k^2) \xi_k^{\prime\prime}(y)+k^2(\lambda+\nu k^2)\xi_k(y)=0 \ \mbox{in}\ (0,1)\\
&&\xi_{k}(0)=\xi_{k}(1)= \xi_k^{\prime}(0)=\xi_k^{\prime}(1)=0.  \nonumber
\end{eqnarray}

The eigenvalue problem for $(\phi_k , \xi_k , q_k )$ when $k= 0$ is
\begin{eqnarray}\label{eqef0}
 \lambda\phi_0(y)-\nu \phi_0^{\prime\prime}(y)&=&0, \nonumber\\
 \lambda\xi_0(y)-\nu \xi_0^{\prime\prime}(y)&=&q_0^{\prime}(y),\nonumber\\
\xi_0^{\prime}(y)&=&0,\\
 \phi_0(0)=\phi_0(1)=\xi_0(0)=\xi_0(1)&=&0.  \nonumber
\end{eqnarray}

We have
$$\xi_0 = 0, q_0=C, \phi_0(y) = D\sin (n\pi y),
\lambda=-\nu\pi^2 n^2, n\in \N.$$

Since $(D\sin(n\pi y),0,0)$ is a solution of the eigenvalue problem \eqref{eqef0} for eigenvalue $\lambda= -\nu\pi^2 n^2, n\in \N$, 
the solution of \eqref{eq:1.2} with control zero and initial condition 
$(D\sin (n\pi y), 0)$, for any $n\in\N,$
is 
$$(u(x,y,t), v(x,y,t), p(x,y,t))=e^{-\nu \pi^2 n^2 t}\big(D\sin (n\pi y), 0, C\Big).$$
Thus, the solution
is exponentially decaying
at the rate $-\nu\pi^2$.
But we cannot get any arbitrary decay for the system \eqref{eq:LS}
using only the normal control $\psi$ and the mode for $k=0$ is not null controllable. To control it, we would
require some additional tangential control.

The following lemma summarizes some elementary facts about the eigenvalues and eigenfunctions for nonzero $k$.

\begin{lemma}\label{lem+ev}
We have the following results regarding the spectrum of the linear operator associated to \eqref{eqefk} and its eigenfunctions. 
\begin{enumerate}
\item The spectrum of the linear operator associated to \eqref{eqefk} is real. 
The resolvent of the linear operator associated to \eqref{eqefk} is compact and hence its spectrum
consists of a set of isolated eigenvalues. 

 \item 
If $\lambda\geq-\nu k^2$, for all $k\in \frac{2\pi}{L}\Z, k\neq 0$, then $\xi_k=0$ for all $k\neq0$.  
Thus, the spectrum of the linear operator associated to \eqref{eqefk} is a subset of $(-\infty,-{4\pi^2\over L^2}\nu)$ and in particular, eigenvalues for the $k$th mode satisfy 
$\lambda_k< -\nu k^2$ for all $k\in \frac{2\pi}{L}\Z-\{0\}$.

In the following, let $l$ be a natural number which counts the eigenvalues for fixed $k$ in order of increasing magnitude. Since we have $\lambda_{k,l}<-\nu k^2$, we may set
$\widetilde\mu_{k,l}=\sqrt{-(k^2+\frac{\lambda_{k,l}}{\nu})}$, and we have $\widetilde\mu_{k,l}\in \R^+$.
\item For all $k\in \frac{2\pi}{L}\Z-\{0\}$, $\{\lambda_{k,l}, \phi_{k,l}, \xi_{k,l}, q_{k,l}\}_{l\in \N}$ is the solution 
of the eigenvalue problem \eqref{eqefk}, where 
\begin{eqnarray}\label{eqefxi}\xi_{k, l} (y)&=& C_1(\lambda_{k,l})e^{ky}+C_2(\lambda_{k,l})e^{-ky}
+C_3(\lambda_{k,l})e^{\mu_{k,l}y}\nonumber\\
&&+C_4(\lambda_{k,l}) e^{-\mu_{k,l}y},\end{eqnarray}
and, for all $k\in \frac{2\pi}{L}\Z-\{0\}$ and for all $l\in \N$, $\lambda_{k,l}$ satisfies 
$$\det
\begin{pmatrix}
 1 & 1 & 1 & 1 \\
 k & -k & \mu_{k,l}& -\mu_{k,l}\\
 e^{k} & e^{-k} & e^{\mu_{k,l}} & e^{-\mu_{k,l}} \\
 k e^{k} & -ke^{-k} &\mu_{k,l} e^{\mu_{k,l}}  &  -\mu_{k,l} e^{-\mu_{k,l}} 
\end{pmatrix}=0,$$
where $\mu_{k,l}=i\widetilde\mu_{k,l}$,
and $\lambda_{k,l}< -\nu k^2$ is necessary for a nontrivial solution.

Further, $\{\phi_{k,l}, \xi_{k,l}\}_{l\in \N}$ is an orthogonal family in $(L^2(0,1))^2$. 

\item $\widetilde \mu_{k,l}$ satisfies
\begin{eqnarray}\label{Eq-root-mutilde}
[\sin(\widetilde \mu_{k,l})\sinh(k)]\widetilde \mu^2_{k,l}- 2 k \widetilde \mu_{k,l}[1-\cosh(k)\cos(\widetilde \mu_{k,l})]
\nonumber\\-k^2\sin(\widetilde \mu_{k,l})\sinh(k)=0.
\end{eqnarray}

\end{enumerate}
\end{lemma}

In the following lemma we give some important properties of the positive, real roots of \eqref{Eq-root-mutilde}. (In addition, \eqref{Eq-root-mutilde} has negative real roots which lead to same eigenvalues $\lambda$, and a root at zero, which does not lead to an eigenvalue but is due to the fact that the functions in \eqref{eqefxi} fail to be linearly independent in that case).
\begin{lemma}\label{lemevprop}
The solution of \eqref{Eq-root-mutilde} for all $k\in \frac{2\pi}{L}\Z-\{0\}$ satisfies:
\begin{enumerate}
\item
For any small $\delta>0$, there exists a (sufficiently large) $k_0\in \N$, such that 
for all $k\in \frac{2\pi}{L}\Z-\{0\}$ with $|k|\ge k_0$,
we have the following: 
\begin{itemize}
 \item[a)]  
If any root  $\widetilde{\mu}_{k,j}$ of \eqref{Eq-root-mutilde}, with $|k|\ge k_0$ and some $j\in \N$, satisfies
  $\Big|\frac{\widetilde \mu^2_{k,j}}{k^2}-1\Big|<\delta$,
then it is unique between 
 two consecutive zeros of $\sin(\cdot)$. 
 \item[b)]If the root $\widetilde{\mu}_{k,j}$ of \eqref{Eq-root-mutilde}, with $|k|\ge k_0$ and some $j\in \N$
 satisfies
 $\Big|\frac{\widetilde \mu^2_{k,j}}{k^2}-1\Big|\ge \delta$, then it is unique 
 between two consecutive zeros of $\cos(\cdot)$.

 \item[c)] Moreover, $k_0$ can be chosen large enough such that the gap between two consecutive roots of \eqref{Eq-root-mutilde} 
 for $k\in \frac{2\pi}{L}\Z-\{0\}$ with $|k|\ge k_0$  is always greater than $\pi-\epsilon_0$, 
 for some positive small constant $\epsilon_0>0$. 
\item[d)] There exists a unique root $\widetilde{\mu}_{k,l}$ in $(l\pi, (l+1)\pi)$, for all $l\in \N$. 
 \end{itemize}
\item 
For all $k\in \frac{2\pi}{L}\Z-\{0\}$, $k\neq 0$ and $|k|<k_0$, 
there exist $l_k\in \N$ and $N_k\in \N$, where  $N_k$ is the number of roots of \eqref{Eq-root-mutilde} in $(0, (l_k+1)\pi-\frac{\pi}{4})$, and the following hold:
       \begin{itemize}
		\item[a)] for all $l\ge l_k+1$, there exists a unique root $\widetilde \mu_{k,j}$ of \eqref{Eq-root-mutilde} in the ball $B(l \pi, \allowbreak\pi/4)$, 
		where $j=l+N_k-l_k$ and the root in fact
		lies in the interval $(l \pi - \pi/4, l \pi + \pi/4)$.
		\item[b)] there is no root $\widetilde \mu$ of \eqref{Eq-root-mutilde} 
		between $\widetilde \mu_{k,j}$ and $\widetilde \mu_{k,j+1}$, where $j=l+N_k-l_k$, for $l \geq l_k+1$.
                \item[c)] $\widetilde \mu_{k,l+N_k-l_k} - l \pi \to 0$  as $l \to \infty$.
	\end{itemize}

\item For each fixed $k\in\frac{2\pi}{L}\Z-\{0\}$ and for all $l\in \N$, $\widetilde{\mu}_{k,l}$, the root of \eqref{Eq-root-mutilde} 
is simple.

\item For any $k\in \frac{2\pi}{L}\Z-\{0\}$, there exists a $l_0\in \N$, independent of $k$, such that 
$ \widetilde{\mu}_{k,j}> j\pi/4$, for all $j>l_0$. 

\item 
The solution $\widetilde \mu_{k,l}$ of \eqref{Eq-root-mutilde} corresponds to
$\frac{\lambda_{k,l}}{\nu} = - \widetilde \mu_{k,l}^2 -k^2 $ for all $k\in \frac{2\pi}{L}\Z-\{0\}$ and for all $l\in \N$.
For each fixed $k\in \frac{2\pi}{L}\Z-\{0\}$, and for all $l\in \N$, 
$\lambda_{k,l}$ is a solution of the eigenvalue problem \eqref{eqefk} with multiplicity one and 
there exist positive constants $C_1$ and $C$ independent of $k,l$, such that 
$$\inf_{k,l}\, \{\lambda_{k,l}-\lambda_{k,l+1}\}> C>0, \quad \mathrm{and} \quad 
\sum_{l>l_0} \frac{1}{(-\lambda_{k,l})}< C_1\sum_{l>l_0}\frac{1}{l^2}< \infty,  
$$
where $l_0$ is introduced in the previous property. 

\end{enumerate}

\end{lemma}

\begin{proof}
\begin{enumerate}
 \item[1.a)] 
We consider the following rearrangement of (\ref{Eq-root-mutilde}):
\begin{equation}\label{Eq-root-mutilde-new1}
f(\widetilde \mu)= \frac{1}{\cosh(k)}-\cos(\widetilde \mu)-\frac{\sinh(k)}{\cosh(k)}\sin(\widetilde \mu)
\frac{\frac{\widetilde \mu^2}{k^2}-1}{2\frac{\widetilde \mu}{k}}.
\end{equation}
We see that between any two consecutive zeros of $\sin(\cdot)$, $f$ changes its sign, since, for all $l\in \N$, we have 
$$f((l+1)\pi)= \frac{1}{\cosh(k)}-\cos((l+1)\pi), \quad f(l\pi)= \frac{1}{\cosh(k)}-\cos(l\pi),$$
and $\frac{1}{\cosh(k)}$ is small if $|k|\ge k_0$, and $k_0$ is chosen large enough. 
Hence, there exists a root of \eqref{Eq-root-mutilde-new1} between two consecutive zeros of $\sin(\cdot)$. 

Let us assume that the root of \eqref{Eq-root-mutilde}, denoted by $\widetilde{\mu}_{k,j}$, for some $j\in \N$, 
satisfies
 $\Big|\frac{\widetilde \mu^2_{k,j}}{k^2}-1\Big|<\delta$, where $\delta>0$ is small enough, and 
 $k\in \frac{2\pi}{L}\Z-\{0\}$ satisfying $|k|\ge k_0$, where $k_0$ is large enough chosen later. 
We claim that, for large enough $k_0$ and small enough $\delta$, this root is unique.
For any $\epsilon>0$, choosing $k_0$ large enough and $\delta$ small enough, we get that $\mu$, any zero of $f(\cdot)$ satisfying
$\Big|\frac{\mu^2}{k^2}-1\Big|<\delta$, 
obeys $|\cos(\mu)|<\epsilon$, and  hence this $\mu$ must belong to a small neighbourhood $N_r$ (with radius $0<r<\pi/4$)
of the zero of $\cos(\cdot)$ between two consecutive zeros of $\sin(\cdot)$. Now, if there are multiple zeros of $f(\cdot)$ between 
two consecutive zeros of $\sin(\cdot)$ satisfying $\Big|\frac{\mu^2}{k^2}-1\Big|<\delta$, 
the zeros of $f(\cdot)$ will be in the neighbourhood $N_r$ and between of two zeros of $f(\cdot)$, there must be 
a zero of $f'(\cdot)$ in $N_r$. 
Further, we check 
$$ f'(\mu)= \sin(\mu)-\frac{\sinh(k)}{\cosh(k)}\Big[\cos(\mu)\Big(\frac{\frac{\mu^2}{k^2}-1}{2\frac{\mu}{k}}\Big)
+\frac{\sin(\mu)}{k}\Big(1-\frac{\frac{\mu^2}{k^2}-1}{2\frac{\mu^2}{k^2}}\Big)\Big],$$
and the zeros of $f'(\cdot)$ are in a small neighbourhood 
of the zeros of $\sin(\cdot)$ due to the choice of $k_0$ and $\delta$ and we can make that the neighborhoods 
around the zeros of $\sin(\cdot)$ and $N_r$ disjoint by a suitable refinement of $k_0$ and $\delta$, if necessary.
Thus, $f'(\cdot)$ cannot have any zeros in $N_r$ and hence there exists unique zero of $f(\cdot)$ between two consecutive zeros of $\sin(\cdot)$.

\item[1.b)] Now let the root $\widetilde{\mu}_{k,j}$ for some $j\in \N$ and for all $|k|\ge k_0$, where $k_0$ is large enough, satisfy
$|\frac{\widetilde \mu^2_{k,j}}{k^2}-1|\ge\delta$, where 
$\delta$ is as mentioned above. 
At the two consecutive zeros of $\cos(\cdot)$, from \eqref{Eq-root-mutilde-new1}, it follows that  
for $j=m-1, m$, where $m\in\N$,
$$f((2j+1)\pi/2)=\frac{1}{\cosh(k)}-\frac{\sinh(k)}{\cosh(k)}\sin((2j+1)\pi/2)
\frac{\frac{(2j+1)^2\pi^2/4}{k^2}-1}{2\frac{(2j+1)\pi/2}{k}}.$$
Choosing $k_0$ large enough in the above relation such that for all $|k|\ge k_0$, $\frac{1}{\cosh(k)}$ and $\pi/k$ are small enough, we get that 
$f$ changes sign between two consecutive zeros of $\cos(\cdot)$. 
To prove that $\widetilde{\mu}_{k,j}$ satisfying $|\frac{\widetilde \mu^2_{k,j}}{k^2}-1|\ge\delta$ is the unique root of \eqref{Eq-root-mutilde}
between to consecutive zeros of $\cos(\cdot)$,
we notice that any nonzero roots of \eqref{Eq-root-mutilde} satisfy
\begin{equation}\label{Eq-root-multitude-new2}
\frac{\sin(\mu)(\sinh(k)/\cosh(k))}{(1/\cosh(k))-\cos(\mu)}-2\frac{\mu/k}{(\mu^2/k^2)-1}=0,
\end{equation}
and for any  $\epsilon>0$, there exists $k_0$, large enough, such that for all $|k|\ge k_0$, between to 
consecutive zeros of $\cos(\cdot)$, the zeros of \eqref{Eq-root-multitude-new2} satisfy
\begin{equation}\label{Eq-root-multitude-new3}
\Big|\tan(\mu)+2\frac{\cosh(k)}{\sinh(k)}\frac{\mu/k}{(\mu^2/k^2)-1}\Big|=\Big|\tan(\mu)+\frac{\sin(\mu)}{(1/\cosh(k))-\cos(\mu)}\Big|
<\epsilon.
\end{equation}
Now, if \eqref{Eq-root-multitude-new2} has multiple roots between two consecutive zeros of $\cos(\cdot)$, as argued in the first part, 
$f'(\cdot)$ should have zeros between any two consecutive zeros of $f(\cdot)$. From the representation of $f'(\cdot)$, we see that for $\mu$, any
zero of $f'(\cdot)$, $\tan(\mu)$ satisfies
\begin{equation}\label{Eq-root-multitude-new4}
\Big|\tan(\mu)-\frac{\sinh(k)}{\cosh(k)}\frac{\mu^2/k^2-1}{2\mu/k}\Big|<\epsilon, \quad \forall\, |k|\ge k_0
\end{equation}
for any arbitrary $\epsilon>0$, choosing $k_0$ large enough. In any interval between two consecutive roots of $\cos(\cdot)$, there is one subinterval where
$\tan\mu+\frac{\cosh(k)}{\sinh(k)}\frac{2\mu/k}{\mu^2/k^2-1}$ is small. But within that subinterval, $ \tan\mu-\frac{\sinh(k)}{\cosh(k)}\frac{(\mu^2/k^2)-1}{2\mu/k}$ cannot also be small.
Hence, \eqref{Eq-root-mutilde} has a unique zero between two consecutive zeros of $\cos(\cdot)$.

\item[1.c)] 
Let us assume that  $|k|\ge k_0$ and $k_0$ is large enough.   
If $\widetilde{\mu}_{k,j}$, the root of \eqref{Eq-root-mutilde} is such that 
$|\widetilde{\mu}_{k,j}/k|$ is close to $1$, then the roots are close to the zeros of $\cos(\cdot)$ and hence they are approximately $\pi$ apart.
If $|\mu/k|$ is not close to 1, the roots are close to those of $\tan\mu=-2\frac{\cosh(k)}{\sinh(k)}\frac{(\mu/k)}{(\mu^2/k^2-1)}$. For large $k$, $\cosh k/\sinh k$ is close to 1, and $\mu/k$ changes slowly with $\mu$. Hence $\tan\mu$ changes little between successive roots. Therefore, the roots in this case are also spaced approximately $\pi$ apart.

\item[1.d)]
Since, for all $|k|\ge k_0$, where $k_0$ is large enough, $f$, defined in \eqref{Eq-root-mutilde-new1}, changes sign between two consecutive zeros of $\sin(\cdot)$, 
there can be an odd number of roots of \eqref{Eq-root-mutilde} between two consecutive zeros of $\sin(\cdot)$. 
Now, by $1.c)$, we have that the gap  between two consecutive
roots of \eqref{Eq-root-mutilde-new1} is greater than $\pi-\epsilon_0$. Choosing $0<\epsilon_0<\pi/2$, 
we can derive that each interval
$(l\pi, (l+1)\pi)$, for all $l\in \N$, contains a unique root of \eqref{Eq-root-mutilde} and we denote the root by $\widetilde{\mu}_{k,l}$. This argument does not apply to the interval $(0,\pi)$, since at $\tilde\mu=0$, there is a zero in the denominator of the last term in \eqref{Eq-root-mutilde-new1}. For large $k$, the dominant terms in $f$
are
$$-\cos(\tilde\mu)+k{\sin(\tilde\mu)\over2\tilde{\mu}}.$$
The second term in this expression is positive and it dominates except near $\tilde\mu=\pi$, where $-\cos\tilde\mu$ is also positive. Hence the expression is strictly positive
on the entire interval $[0,\pi]$, and it is easy to show that for large $k$ the perturbing terms do not change that. Hence there is no root between $0$ and $\pi$.

\item[2.]For all $k\in \frac{2\pi}{L}\Z-\{0\}$, $k\neq 0$ and $|k|<k_0$, 
we get that there exists a positive natural number $l_k$ such that 
$$ 
\begin{array}{l}
\Big|- 2 k z[1-\cosh(k)\cos(z)]
\\-k^2\sin(z)\sinh(k)\Big|
\le  \Big|[\sin(z)\sinh(k)]z^2\Big|, \quad \forall \, z\in \pa B(l\pi, \pi/4), \forall \, l>l_k.
\end{array}
$$
This inequality holds for large enough $l$ since $|z^2|>>|z|$, and $|\sin z/\cos z|$ is bounded below on the periphery of the circle.
Comparing the solutions of \eqref{Eq-root-mutilde} with the solutions of $[\sinh(k)\sin(z) ] z^2$, by Rouch\'e's theorem,
we obtain that for all $k\in \frac{2\pi}{L}\Z-\{0\}$, $k\neq 0$ and $|k|<k_0$ and for all $l > l_k$, there exists a unique solution 
of \eqref{Eq-root-mutilde} in the ball $B(l \pi, \allowbreak\pi/4)$, which in fact lies in the interval $(l \pi - \pi/4, l \pi + \pi/4)$ 
as Property $1$. of Lemma \ref{lem+ev} gives that all roots of \eqref{Eq-root-mutilde} are real. 
The other results also can be derived by  comparing the solution of \eqref{Eq-root-mutilde} for all $k\in \frac{2\pi}{L}\Z-\{0\}$, $k\neq 0$ 
and $|k|<k_0$ and for all $l>l_k$
with the zeros of $[\sinh(k)\sin(z) ] z^2$.

\item[3.] For a $k\in \frac{2\pi}{L}\Z-\{0\}$, 
if $\widetilde{\mu}$ is a multiple root of \eqref{Eq-root-mutilde}, then there exist two solutions $\xi_{k,1}$ and $\xi_{k,2}$ of 
the eigenvalue problem \eqref{eqeval1} for the eigenvalue $\widetilde{\mu}$. By choosing the constants $d_1$ and $d_2$, appropriately,
we get that $\xi= d_1\xi_{k,1}+d_2\xi_{k,2}$, satisfying $\xi^{''}(0)= 0$. Since, $\xi$ satisfies \eqref{eqeval1} for the same $\widetilde{\mu}$, 
by Property 3. in Lemma \ref{lem+ev}, we have that 
$$ \xi(y)= a_1\cosh(ky)+ b_1\cos(\widetilde\mu y)+a_2 \sinh(ky)+ b_2 \sin(\widetilde\mu y), \quad \forall\, y\in (0,1).$$
Using $\xi(0)=0= \xi^{''}(0)$, we get that for all $y\in (0,1)$,
$ \xi(y)=a_2 \sinh(ky)+ b_2 \sin(\widetilde\mu y)$. 
From $\xi'(0)=0=\xi'(1)$, it can be derived that $\cosh(k)=\cos(\widetilde \mu)$.
This is a contradiction because of $\cosh(k)>1$ for all $k\neq 0$ and  $\cos(\widetilde \mu)\le 1$. 
Thus, $\widetilde{\mu}$ has to be a simple root of \eqref{Eq-root-mutilde}.

\item[4.] 
From $1.d)$ we have that 
if $k_0$ is large enough and $|k|\ge k_0$, then 
$\widetilde{\mu}_{k,j}$, the solution of \eqref{Eq-root-mutilde}, belongs to $(j\pi, (j+1)\pi)$ and hence $\widetilde{\mu}_{k,j}> j\pi/2$, for all 
$j\ge 1$ and $|k|\ge k_0$. 

Now for each $k\in \frac{2\pi}{L}\Z-\{0\}$ and $|k|<k_0$, from $2. a)$, it follows that $\widetilde{\mu}_{k,j}$ belongs to 
$((j-N_k+l_k)\pi-\pi/4, (j-N_k+l_k)\pi+\pi/4)$ and so
$\widetilde{\mu}_{k,j}> (j-N_k+l_k)\pi-\pi/4$, where $N_k$ and $l_k$ are introduced 
in $2.a)$. Let us choose  $N_0=\max\{N_k \mid |k|<k_0, \quad k\neq 0\}$. Then, $k\in \frac{2\pi}{L}\Z-\{0\}$ and $|k|<k_0$ and for all $j>2N_0$, 
we get that 
$$\widetilde{\mu}_{k,j}> (j-N_k+l_k)\pi-\pi/4> (j-N_0)\pi-\pi/4> j\pi/2-\pi/4>j\pi/4.$$
Choosing $l_0= 2N_0$, we get our result.

\item[5.] 
Using all the above results, the claim follows.

\end{enumerate}

\end{proof}

\begin{lemma}\label{lemef}
Let us recall that $\{\phi_{k,l}, \xi_{k,l}, q_{k,l}\}_{l\in \N}$ is the solution 
of the eigenvalue problem \eqref{eqefk} corresponding to the eigenvalue $\lambda_{k,l}$. We have the following explicit expression for the eigenfunctions $(\phi_{k,l}, \xi_{k,l}, q_{k,l})$: 
\begin{enumerate}
 \item The coefficients in the expression of $\xi_{k,l}$ are
\begin{eqnarray}\label{eqcoeffef}
C_1(\lambda_{k,l})& =& \mu^2_{k,l}\Big[e^{-(\mu_{k,l}+k)}-e^{(\mu_{k,l}-k)}\Big] + \mu_{k,l}\Big[2k-k (e^{-(\mu_{k,l}+k)}+e^{(\mu_{k,l}-k)}\Big], \nonumber\\
C_2(\lambda_{k,l})&=&\mu^2_{k,l}\Big[e^{(\mu_{k,l}+k)}-e^{-(\mu_{k,l}-k)}\Big] + \mu_{k,l}\Big[2k-k (e^{-(\mu_{k,l}-k)}+e^{(\mu_{k,l}+k)}\Big], \nonumber\\
C_3(\lambda_{k,l})&=& \mu_{k,l}\Big[2k -k \Big(e^{-(\mu_{k,l}+k)}+ e^{-(\mu_{k,l}-k)}\Big)\Big]+ k^2\Big[ e^{-(\mu_{k,l}+k)}-e^{-(\mu_{k,l}-k)}\Big], \nonumber\\
C_4(\lambda_{k,l})&=& \mu_{k,l}\Big[2k -k \Big(e^{(\mu_{k,l}+k)}+ e^{(\mu_{k,l}-k)}\Big)\Big] + k^2\Big[ e^{(\mu_{k,l}+k)}-e^{(\mu_{k,l}-k)}\Big].\nonumber\\
\end{eqnarray} 
\item For all $k\in \frac{2\pi}{L}\Z-\{0\}$ and $l\in \N$, $q_{k,l}(1)= -\frac{\nu}{k^2} \xi'''_{k,l}(1)$, where 
$$\xi'''_{k,l}(1)= -i4k \frac{\lambda_{k,l}}{\nu} \Big[ \widetilde{\mu}_{k,l}\Big\{ \frac{2k}{\sinh(k)}(1-\cosh(k)\cos(\widetilde{\mu}_{k,l}))+ k \sinh(k)\Big\}$$
$$+k^2\sin(\widetilde{\mu}_{k,l})\Big].
$$ 
\end{enumerate}
\end{lemma}

\begin{lemma}\label{lemexp}
Let $(\phi, \xi)\in  \mathbf{V}_{\sharp,n}^0(\Omega)$. Then we have 
$$ \Big(\begin{array}{l} \phi(x,y) \\ \xi(x,y) \end{array}\Big)= 
\sum_{k\in \frac{2\pi}{L}\Z}\sum_{l\in \N} \alpha_{k,l}\Big( \begin{array}{l} \phi_{k,l}(y) \\ \xi_{k,l}(y) \end{array}\Big)e^{ikx}, 
$$
where $\{\phi_{k,l}, \xi_{k,l}\}_{k\in \frac{2\pi}{L}\Z, l\in \N}$ are the eigenfunctions associated to the eigenvalue problem \eqref{eqefk}-\eqref{eqef0}. 
\end{lemma}

We have the following existence theorem for the solution of \eqref{eq:LS}.
\begin{theorem}\label{thexist}
Let $T>0$. 
For any $U_0\in  \mathbf{V}_{\sharp,n}^0(\Omega)$ with 
$\displaystyle \int_0^LU_0(x,y)\,dx= 0$, for all $y\in (0,1)$ and for any $\psi\in L^2(0,T;\dot{L}^2(0,L))$, system \eqref{eq:LS} admits a unique solution
$U$ in $C([0,T]; \mathbf{V}^{-1}_\sharp(\Omega))$ with $\displaystyle \int_0^L U(x,y,t)\,dx= 0$, for all
$(y,t)\in(0,1)\times (0,T)$.

Further, if $\psi$ vanishes near $t=T$, then the solution of \eqref{eq:LS} at $t=T$, $U(T)$ is smooth and 
in particular in $ \mathbf{V}_{\sharp,n}^0(\Omega)$.
\end{theorem}

\begin{remark}\hglue 1in
\begin{enumerate}
\item We note that $\int_0^L v_0(x,y)\,dx$ is automatically zero due to the boundary condition at the bottom wall and the incompressibility condition. Hence the condition on the average of $U_0=(u_0,v_0)$ is really just a condition on $u_0$.
\item We do not claim that the regularity of the solution as stated in the preceding theorem is optimal (for a more careful discussion of regularity for solutions of inhomogeneous Stokes problems, see \cite{Raymond}). Moreover, the choice
of function spaces is not essential. If we leave our system uncontrolled for a short period of time, the solution will become
infinitely smooth, so we could have assumed in the first place that our initial data are smooth and then chosen a control which is
also smooth.
\item If we only know $U\in C([0,T],\mathbf{V}^{-1}_\sharp(\Omega))$, we can interpret the vanishing of the integral over $x$ as the vanishing of the $k=0$ Fourier component. However, for $t>0$ and $y<1$, $U$ is actually of class $C^\infty$, and the integral is defined in the classical sense.
\end{enumerate}
\end{remark}

\section{Null controllability}\label{sec3}
In this section, we study the null controllability of system \eqref{eq:LS}. 
In particular, we have the following result:
\begin{theorem}\label{thnull}
Let $T>0$. For any $U_0\in \mathbf{V}_{\sharp,n}^0(\Omega)$ with 
$\displaystyle \int_0^LU_0(x,y)\,dx= 0$, for all $y\in (0,1)$, there exists a control $\psi\in L^2(0,T; \dot{L}^2(0,L))$, 
such that the solution of \eqref{eq:LS} reaches zero at $t=T$. 
\end{theorem}

To prove the above theorem, it is enough to show the following result holds.
\begin{proposition}\label{propidenty}
Let $T>T_0>0$.
Let us assume that there exists an operator $B\in \mathcal{L}(L^2(0,T_0; \dot{L}^2(0,L)), \mathbf{L}^2_\sharp(\Omega))$ such that
\begin{equation}\label{cntrl1}
B(q(\cdot,1,\cdot)|_{(0,L)\times(0,T_0)})= \Phi(\cdot, \cdot,0), 
\end{equation}
where $(\Phi, q)$ is the solution of the adjoint problem \eqref{eqadj} with terminal condition $\Phi_T\in \mathbf{V}_{\sharp,n}^0(\Omega)$ with 
$\displaystyle \int_0^L \Phi_T(x,y)\,dx= 0$, for all $y\in (0,1)$. 
Set 
\begin{equation}\label{eqcntrl1}
\psi= \left\{\begin{array}{l} B^*U_0 \quad \mathrm{in} \quad (0,L)\times (0,T_0), \\ 0 \quad \mathrm{in} \quad (0,L)\times [T_0,T). \end{array}\right.
\end{equation}
Then, for any $T>0$ and for any $U_0\in \mathbf{V}_{\sharp,n}^0(\Omega)$ with 
$\displaystyle \int_0^LU_0(x,y)\,dx= 0$, for all $y\in (0,1)$, $U$, the solution of \eqref{eq:LS} with control 
$\psi$ defined in \eqref{eqcntrl1} and initial condition $U_0$, 
reaches zero at $t=T$. 
\end{proposition}
\begin{proof}
Let us note that since $\psi$, defined in \eqref{eqcntrl1}, vanishes near $T$, 
from the last part of Theorem \ref{thexist}, it follows that at $t=T$, 
$U(\cdot)$, the solution of \eqref{eq:LS} with control $\psi$, is in $\mathbf{L}^2_\sharp(\Omega)$ (in fact $C^\infty$ smooth). 
Multiplying \eqref{eq:LS} with $\Phi$, the solution of the adjoint problem \eqref{eqadj} with terminal condition $\Phi_T\in \mathbf{V}_{\sharp,n}^0(\Omega)$ with 
$\displaystyle \int_0^L \phi_T(x,y)\,dx= 0$, for all $y\in (0,1)$, and using an integration by parts, we obtain the identity
\begin{eqnarray}\label{eqIdentity}
\int_0^T \int_0^L \psi(x,t) q(x,1,t)\, dx\, dt=\\  \langle U_0(\cdot, \cdot), \Phi(\cdot, \cdot,0)\rangle_{\mathbf{L}^2_\sharp(\Omega)}-\langle U(\cdot, \cdot,T), \Phi_T(\cdot, \cdot)\rangle_{\mathbf{L}^2_\sharp(\Omega)}. \nonumber
\end{eqnarray}
Now, using $\psi$ defined in \eqref{eqcntrl1}, from \eqref{eqIdentity}, we obtain 
$$ \langle U(\cdot, \cdot,T), \Phi_T(\cdot, \cdot)\rangle_{\mathbf{L}^2_\sharp(\Omega)}=0, \quad \forall \, \Phi_T\in \mathbf{V}_{\sharp,n}^0(\Omega),$$
and hence $U(\cdot, \cdot, T)=0$. 
\end{proof}
\begin{remark}
\begin{enumerate}
\item We note that only the pressure $q$ appears in the normal component of stress. The viscous stress vanishes as a result of the divergence condition.
\item We have introduced $T_0<T$ only to avoid any technical issues related to lack of regularity of the solution of the adjoint equation. By choosing $T_0<T$, we ensure that the solutions of the adjoint problem are in fact $C^\infty$ smooth for $t\in [0,T_0]$. Therefore, at any time, either $U$ or $\Phi$ is $C^\infty$ smooth in the preceding proposition. Choosing $T>T_0$ also guarantees that the integral of $\|q\|^2$ in the following proposition is finite. We note that a posteriori at is clear that $U$ in Proposition \ref{propidenty} actually vanishes at $T_0$, since backward uniqueness holds for the Stokes equations.
\end{enumerate}
\end{remark}
Next we prove the existence of the bounded operator $B$, defined in \eqref{cntrl1}, by showing that 
the observability inequality associated to the null control problem of \eqref{eq:LS} holds. 
\begin{proposition}\label{propobs}
Let us assume that $B$ is as defined by (\ref{cntrl1}). 
For any $T>T_0>0$, $B\in \mathcal{L}(L^2(0,T_0; \dot{L}^2(0,L)), \mathbf{L}^2_\sharp(\Omega))$, i.e, there exists a positive constant $C(T_0)>0$, such that 
\begin{equation}\label{eqobs}
 \displaystyle
\int_0^{T_0} \|q(\cdot,1,t)\|^2_{\dot{L}^2(0,L)}\, dt \geq C(T_0) \|(\phi, \xi)(\cdot,\cdot,0)\|^2_{\mathbf{L}^2_{\sharp}(\Omega)},
\end{equation}
where $(\phi, \xi,q)$ is the solution of \eqref{eqadj} with terminal condition $(\phi_T, \xi_T)\in \mathbf{V}_{\sharp,n}^0(\Omega)$ satisfying 
$\displaystyle \int_0^L \phi_T(x,y)\, dx\allowbreak=0, \quad \forall\, y\in (0,1)$. 
\end{proposition}

Let us consider the series expansion of $(\phi, \xi, q)$ 
\begin{eqnarray}\label{eqexp1}
&&\phi(x,y,t)= \sum_{k\in \frac{2\pi}{L}\Z-\{0\}} \phi_k(y,t) e^{ikx}, \quad \xi(x,y,t)= \sum_{k\in \frac{2\pi}{L}\Z-\{0\}} \xi_k(y,t) e^{ikx}, \nonumber\\
&&q(x,y,t)= \sum_{k\in \frac{2\pi}{L}\Z-\{0\}} q_k(y,t) e^{ikx}, \\
&&\phi_T(x,y)= \sum_{k\in \frac{2\pi}{L}\Z-\{0\}} \sum_{l\in\N} \alpha_{k,l}\phi_{k,l}(y) e^{ikx},\nonumber\\
&&\xi_T(x,y)= \sum_{k\in \frac{2\pi}{L}\Z-\{0\}} \sum_{l\in \N} \alpha_{k,l} \xi_{k,l}(y) e^{ikx},\nonumber
\end{eqnarray}
where $(\phi_{k,l}, \xi_{k,l})$ is the solution of the eigenvalue problem \eqref{eqefk} for all $k\in \frac{2\pi}{L}\Z-\{0\}$ and $l\in \N$.

Using the orthogonality of $\{e^{ikx}\}_{k\in \frac{2\pi}{L}\Z-\{0\}}$, from \eqref{eqadj} we can derive that for all $k\in \frac{2\pi}{L}\Z-\{0\}$, 
$(\phi_k, \xi_k, q_k)$ satisfies
\begin{eqnarray}\label{eqadjsysk}
 -\pa_t \phi_k(y,t)=\nu[-k^2 \phi_k(y,t)+ \phi_k^{''}(y,t)] + ik q_k(y,t), \quad \forall \, (y,t)\in (0,1)\times (0,T), \nonumber\\
 -\pa_t \xi_k(y,t)=\nu[-k^2 \xi_k(y,t)+ \xi_k^{''}(y,t)] + q'_k(y,t), \quad \forall \, (y,t)\in (0,1)\times (0,T),\nonumber\\
ik\phi_k+ \xi'_k = 0, \quad \forall \, (y,t)\in (0,1)\times (0,T), \\
\phi_k(0,t)= \phi_k(1, t)=0=\xi_k(0,t)=\xi_k(1,t), \quad \forall \, t\in (0,T), \nonumber\\
\phi_k(\cdot, T)= \sum_{l\in \N} \alpha_{k,l} \phi_{k,l}(y), \quad \xi_k(\cdot,T)=\sum_{l\in \N} \alpha_{k,l}\xi_{k,l}(y), \quad \forall \, y\in (0,1). \nonumber
\end{eqnarray}
Using this system, to prove the observability inequality \eqref{eqobs}, it is sufficient to show that
the observability inequality for each $k\in \frac{2\pi}{L}\Z-\{0\}$ holds with a positive constant $C(T_0)$, independent of $k$. The following lemma gives
an estimate for the eigenfunctions which turns out to be crucial in establishing this.

\begin{lemma}\label{lemestcntrl}
Let us recall from Lemma \ref{lemef}, the representation of $\xi_{k,l}$ for all $k\in \frac{2\pi}{L}\Z-\{0\}$ and $l\in \N$. 
There exists a positive constant $M$, independent of $k$ and $l$, such that
\begin{equation}\label{eqobsest2} |\xi^{'''}_{k,l}(1)|\ge M k^2e^{|k|} |\lambda_{k,l}||\mu_{k,l}|.\end{equation} 
\end{lemma}
\begin{proof}
From Lemma \ref{lemef}, since we have 
$$\xi^{'''}_{k,l}(1) = -i4k \frac{\lambda_{k,l}}{\nu} \Big[ \widetilde{\mu}_{k,l}\Big\{ \frac{2k}{\sinh(k)}(1-\cosh(k)\cos(\widetilde{\mu}_{k,l}))+ k \sinh(k)\Big\}+k^2\sin(\widetilde{\mu}_{k,l})\Big],
$$
then we get:
\begin{enumerate}
 \item There exists a large $\widehat{k}\in \N$, such that for all $|k|>\widehat{k}$ and for all $l\in \N$, we have 
$$
|\xi^{'''}_{k,l}(1)|\ge M_* k^2e^{|k|} |\lambda_{k,l}||\mu_{k,l}|,
$$
for some $M_*>0$. 

\item For $|k|\le \widehat{k}$, $k\neq 0$, using the fact that $\widetilde{\mu}_{k,l}-(l-N_k+l_k)\pi \rightarrow 0$ as 
$l\rightarrow\infty$ from 2.c) in Lemma \ref{lemevprop}, we find positive constants $\widehat{l}_k$ and $\widehat{M}_k$, independent of $l$, 
such that  
$$
|\xi^{'''}_{k,l}(1)|\ge \widehat{M}_k k^2e^{|k|} |\lambda_{k,l}||\mu_{k,l}|, \quad \forall \, l>\widehat{l}_k.
$$
To see this, we use the following estimate:
\begin{eqnarray}
| \frac{2}{\sinh(k)}(1-\cosh(k)\cos(\widetilde{\mu}_{k,l}))+  \sinh(k)|\ge {2+\sinh^2k-2\cosh k\over|\sinh k|}\nonumber\\={(\cosh k-1)^2\over|\sinh k|}.
\end{eqnarray}

This also shows that $\xi'''_{k,l}(1) \neq 0$ for all $k\in \frac{2\pi}{L}\Z-\{0\}$, and $l\in \N$ (see also Lemma 4.1 in \cite{Munteanu} and Proposition 2.1 in \cite{chowdhury}). Thus
there exists a positive $M_k$ independent of $l$, such that 
$$
|\xi^{'''}_{k,l}(1)|\ge M_k k^2e^{|k|} |\lambda_{k,l}||\mu_{k,l}|, \quad \forall \, l\in \N, \quad \forall \, |k|\le \widehat{k}, \, k\neq 0.
$$

\item Finally, taking $M= \min\{ M_*, M_k, \quad |k|\le \widehat{k}, \, k\neq 0\}$ (which is a positive number),
we obtain that 
$$|\xi^{'''}_{k,l}(1)|\ge M k^2e^{|k|} |\lambda_{k,l}||\mu_{k,l}|, \quad \forall \, k\in \frac{2\pi}{L}\Z-\{0\}, \quad \forall\, l\in \N.$$
\end{enumerate}
\end{proof}

\begin{lemma}\label{lemobsk}
For any $T>T_0>0$, there exists a positive constant $C(T_0)>0$, independent of $k$, such that for all $k\in \frac{2\pi}{L}\Z-\{0\}$, 
\begin{equation}\label{eqobsk}
 \displaystyle
\int_0^{T_0} |q_k(1,t)|^2\, dt \geq C(T_0)\|(\phi_k, \xi_k)(\cdot,0)\|^2_{L^2(0,1)},
\end{equation}
where $(\phi_k, \xi_k,q_k)$ is the solution of \eqref{eqadjsysk} with terminal condition $(\phi_k(\cdot,T), \xi_k(\cdot,T))\in (L^2(0,1))^2$. 
\end{lemma}
\begin{proof}
From \eqref{eqadjsysk}, we can derive that 
$$ \phi_k(y,t)= \sum_{l\in \N} \alpha_{k,l}e^{\lambda_{k,l}(T-t)}\phi_{k,l}(y), \quad 
\xi_k(y,t)= \sum_{l\in \N} \alpha_{k,l}e^{\lambda_{k,l}(T-t)}\xi_{k,l}(y). 
$$
Now using the representation of the eigenfunctions $\{\phi_{k,l}, \xi_{k,l}\}_{l\in \N}$ from Lemma \ref{lem+ev}, 
we obtain that there exists a positive constant $M$ independent of $k$ such that 
\begin{equation}\label{eqobsest1} \|\phi_k(\cdot,0), \xi_k(\cdot, 0)\|^2_{L^2(0,1)} \le M \sum_{l\in \N} |\alpha_{k,l}|^2
e^{2\lambda_{k,l} T} |\lambda_{k,l}|^2e^{2|k|}|\mu_{k,l}|^2.\end{equation}
From $\eqref{eqadjsysk}_1$, using the boundary condition $\phi_k(1)=0$, we obtain
$$ikq_k(1)=-\nu\phi_k''(1),$$
and by combining this with the incompressibility condition, we find
$$q_k(1)=-{\nu\over k^2}\xi_k'''(1).$$
Hence, we get that for all $k\in \frac{2\pi}{L}\Z-\{0\}$, 
$$ q_k(1,t)=-\frac{\nu}{k^2} \sum_{l}\alpha_{k,l} e^{\lambda_{k,l}(T-t)}\xi^{'''}_{k,l}(1),$$ 
and from the expression of $2.$ in Lemma \ref{lemef}, we have that 
$$\xi^{'''}_{k,l}(1) = -i4k \frac{\lambda_{k,l}}{\nu} \Big[ \widetilde{\mu}_{k,l}\Big\{ \frac{2k}{\sinh(k)}(1-\cosh(k)cos(\widetilde{\mu}_{k,l}))+ k \sinh(k)\Big\}+k^2\sin(\widetilde{\mu}_{k,l})\Big]
$$ and from Lemma \ref{lemestcntrl}, we have \eqref{eqobsest2}, i.e.,
\begin{equation*}|\xi^{'''}_{k,l}(1)|\ge M k^2e^{|k|} |\lambda_{k,l}||\mu_{k,l}|, \quad \forall \, k\in \frac{2\pi}{L}\Z-\{0\}, \quad \forall \, l\in \N,
\end{equation*}
for some positive constant $M$.
Since, 
by Lemma \ref{lemevprop}, we have that
$$\mathrm{infimum}_{k,l}\,\{ \lambda_{k,l}-\lambda_{k,l+1}\}>C, \quad \sum_{l>l_0}\frac{1}{(-\lambda_{k,l})}< \sum_{l>l_0}\frac{1}{l^2}<\infty,$$
where $C$ and $l_0$ are independent of $k$, 
by using a M\"untz-Sz\'asz theorem (see \cite{Lopez}, Proposition 3.2), we can show that 
$$ \int_0^{T_0} |\frac{\nu}{k^2} \sum_{l}\alpha_{k,l} e^{\lambda_{k,l}(T-t)}\xi^{'''}_{k,l}(1)|^2\, dt\ge
\int_{T_0/2}^{T_0} |\frac{\nu}{k^2} \sum_{l}\alpha_{k,l} e^{\lambda_{k,l}(T-t)}\xi^{'''}_{k,l}(1)|^2\, dt$$
$$\ge C(T_0) \sum_{l\in \N} |\alpha_{k,l}|^2/(k^4|\lambda_{k,l}|) |\xi^{'''}_{k,l}(1)|^2e^{2\lambda_{k,l}(T-T_0/2)}.
$$ 
From \eqref{eqobsest1} and \eqref{eqobsest2}, it follows that 
$$ \int_0^{T_0} |q_k(1,t)|^2\, dt \ge C(T_0) \|\phi_k(\cdot,0), \xi_k(\cdot, 0)\|^2_{L^2(0,1)},$$
and hence we get \eqref{eqobsk}, since $\exp(-\lambda_{k,l}T_0)/|\lambda_{k,l}|>T_0$.
\end{proof}

\textbf{Proof of Proposition \ref{propobs}}:
\begin{proof}
Note that 
$$ \int_0^{T_0} \|q(\cdot,1,t)\|^2_{\dot{L}^2(0,L)}\, dt = L \sum_{k\in \frac{2\pi}{L}\Z-\{0\}}\int_0^{T_0} |q_k(1,t)|^2\, dt, $$
and 
$$ \|(\phi, \xi)(\cdot,\cdot,0)\|^2_{\mathbf{V}_{\sharp,n}^0(\Omega)}= L \sum_{k\in \frac{2\pi}{L}\Z-\{0\}}\|\phi_k(\cdot,0), \xi_k(\cdot, 0)\|^2_{L^2(0,1)}.$$ 
The result now follows from Lemma \ref{lemobsk}.
\end{proof}

\medskip
Received xxxx 20xx; revised xxxx 20xx.
\medskip

\begin{thebibliography}{99}

\bibitem{Krstic1}
\newblock O.M. Aamo, M. Krstic and T.R. Bewley,
\newblock Control of mixing by boundary feedback in
 2D-channel,  
 \newblock \emph{Automatica J. IFAC}, \textbf{39} (2003), no. 9, 1597--1606, doi: 10.1016/S0005-1098(03)00140-7.
 
 \bibitem{Krstic2}
 \newblock A. Balogh, W.-J. Liu and M. Krstic, 
 \newblock Stability enhancement by boundary control in
 2D channel flow, 
 \newblock \emph{IEEE Trans. Automat. Control}, \textbf{46} (2001), no. 11, 1696--1711, doi: 10.1109/9.964681.

\bibitem{Barbu0}
\newblock V. Barbu, 
\newblock Stabilization of a plane periodic channel flow by noise wall normal controllers, 
\newblock \emph{Systems Control Lett.},
\textbf{59} (2010), no. 10, 608--614, doi:10.1016/j.sysconle.2010.07.005

\bibitem{Barbu1}
\newblock V. Barbu, 
\newblock Stabilization of a plane channel flow by wall normal controllers,
\newblock \emph{Nonlinear Anal.,} {\textbf 67} (2007), no. 9, 2573--2588, doi: 10.1016/j.na.2006.09.024.

\bibitem{Barbu2} 
\newblock V. Barbu, 
\newblock \emph{Stabilization of Navier-Stokes flows,} 
\newblock Communications and Control Engineering Series. Springer, London, 2011, doi: 10.1007/978-0-85729-043-4.

\bibitem{chowdhury} 
\newblock S. Chowdhury and
S. Ervedoza,
\newblock Open loop stabilization of incompressible Navier-Stokes equations in a $2$d channel using power series expansion,
\newblock submitted to \emph{J. Europ. Math. Soc.},
\newblock https://www.math.univ-toulouse.fr/~ervedoza/Publis/Chowdhury-Erv.pdf.

 \bibitem{coron1}
 \newblock J.-M. Coron, 
 \newblock {\emph Control and Nonlinearity},
 \newblock Math. Surveys Monogr. 136, American Mathematical
 Society, Providence, RI, 2007, doi: 10.1090/surv/136.
 
 \bibitem{coron2} 
 \newblock J.-M. Coron and E. Cr\'{e}peau, 
 \newblock Exact boundary controllability of a nonlinear KdV equation
 with critical lengths,
 \newblock \emph{J. Eur. Math. Soc.,} \textbf{6} (2004), no. 3, 367--398, doi: 10.4171/JEMS/13.


\bibitem{coron5}
\newblock J.-M. Coron and S. Guerrero, 
\newblock Local null controllability of the two-dimensional Navier-Stokes
system in the torus with a control force having a vanishing component, 
\newblock \emph{J. Math. Pures Appl.(9)},
\textbf{92} (2009), no. 5, 528--545, doi: 10.1016/j.matpur.2009.05.015.

\bibitem{coron6}
\newblock J.-M. Coron and P. Lissy,
\newblock Local null controllability of the three-dimensional Navier-Stokes system with a distributed control having 
two vanishing components, 
\newblock \emph{Invent. Math.,} \textbf{198} (2014), no. 3,  833--880, doi: 10.1007/s00222-014-0512-5.


\bibitem{Lopez}
\newblock A. Lopez and E. Zuazua, 
\newblock Uniform null-controllability for the one-dimensional heat equation with rapidly oscillating periodic density,
\newblock \emph{Ann. Inst. H. Poincar\'{e} Anal. Non Lin\'{e}aire}, \textbf{19} (2002), no. 5, 543--580, doi: 10.1016/S0294-1449(01)00092-0. 

\bibitem{Munteanu} 
\newblock I. Munteanu, 
\newblock Normal feedback stabilization of periodic flows
in a two-dimensional channel, 
\newblock \emph{J Optim. Theory Appl.}, \textbf{152} (2012), no. 2, 413--438, doi: 10.1007/s10957-011-9910-7.

\bibitem{Munteanu2} 
\newblock I. Munteanu, 
\newblock Tangential feedback stabilization of periodic flows in a 2-D channel,
\newblock \emph{Differ. Integral Equ.}, \textbf{24} (2011), no. 5-6, 469--494.


\bibitem{Raymond} 
\newblock J.-P. Raymond, 
\newblock Stokes and Navier-Stokes equations with nonhomogeneous boundary conditions,
\newblock \emph{Ann. Inst. H. Poincar\'{e} Anal. Non Lin\'{e}aire}, \textbf{24} (2007), no. 6, 921--951, doi: 10.1016/j.anihpc.2006.06.008.

\bibitem{Raymond2}
\newblock J.-P. Raymond, 
\newblock Feedback boundary stabilization of the two-dimensional Navier-Stokes equations,
\newblock \emph{SIAM J. Control Optim.}, \textbf{45} (2006), no. 3, 790--828, doi:10.1137/050628726.


\bibitem{Tri}
\newblock R. Triggiani, 
\newblock Stability enhancement of a 2-D linear Navier-Stokes channel
flow by a 2-D wall-normal boundary controller, 
\newblock \emph{Discrete Contin. Dyn. Syst.
Ser. B}, \textbf{8} (2007), no. 2, 279--314, doi: 10.3934/dcdsb.2007.8.279.

\bibitem{Vaz1}
\newblock R. V\'{a}zquez, E. Tr\'{e}lat and J.-M.Coron, 
\newblock Control for fast and stable laminar-to-high-Reynolds-numbers transfer in a 2D Navier-Stokes channel flow, 
\newblock \emph{Discrete Contin. Dyn. Syst. Ser. B}, {\bf 10} (2008), no. 4, 925--956, doi:10.3934/dcdsb.2008.10.925.

\bibitem{Vaz2}
\newblock R. V\'{a}zquez and M. Krstic, 
\newblock A closed-form feedback controller for stabilization of the linearized 2-D Navier-Stokes Poiseuille system,
\newblock \emph{IEEE Trans.
Automat. Control}, \textbf{52} (2007), no. 12, 2298--2312, doi: 10.1109/TAC.2007.910686 .
\end{thebibliography}
\end{document}